\documentclass{article}
\usepackage{amssymb, amsmath, amsthm, verbatim, amstext, amsfonts, enumerate}
\usepackage[latin1]{inputenc}

\newcommand{\set}[1]{\ensuremath{\left\{#1\right\}}}
\newcommand{\abs}[1]{\ensuremath{|#1|}}

\newcommand{\isom}{\ensuremath{\cong}}
\newcommand{\inv}{\ensuremath{^{-1}}}

\newcommand{\Aut}{\textnormal{Aut}}
\newcommand{\closure}[1]{\overline{#1}}

\newcommand{\sm}{\ensuremath{\setminus}}
\newcommand{\es}{\ensuremath{\emptyset}}

\newcommand{\sub}{\subseteq}
\newcommand{\open}[1]{\mathring{#1}}

\theoremstyle{definition}
\newtheorem{Def}{Definition}[section]
\newtheorem{Exam}[Def]{Example}
\newtheorem{Rem}[Def]{Remark}

\theoremstyle{plain}
\newtheorem{Lem}[Def]{Lemma}
\newtheorem{Cor}[Def]{Corollary}

\newtheorem{Tm}[Def]{Theorem}

\newtheorem{Claim}[Def]{Claim}

\newcommand{\nat}{{\mathbb N}}

\newcommand{\BF}{\ensuremath{\mathcal B}}
\newcommand{\CF}{\ensuremath{\mathcal C}}

\newcommand{\EF}{\ensuremath{\mathcal E}}

\newcommand{\SF}{\ensuremath{\mathcal S}}
\newcommand{\TF}{\ensuremath{\mathcal T}}

\newcommand{\WF}{\ensuremath{\mathcal W}}

\begin{document}

\title{Transitivity conditions in infinite graphs}
\author{Matthias Hamann\and Julian Pott}
\date{Mathematisches Seminar\\Universit\"at Hamburg\\{\small Bundesstr. 55, 20146 Hamburg, Germany}}
\maketitle

\begin{abstract}
We study transitivity properties of graphs with more than one end. We completely classify the distance-transitive such graphs and, for all $k\ge 3$, the $k$-CS-transitive such graphs.
\end{abstract}

\section{Introduction}

A {\em $k$-distance-transitive graph} is a graph $G$ such that for every two pairs $(x_1,x_2)$ and $(y_1,y_2)$ of vertices  with distances $d(x_1,x_2)=d(y_1,y_2)\leq k$ there is an automorphism $\alpha$ of~$G$ with $x_i^\alpha=y_i$ for $i=1,2$, where $x_i^\alpha$ is the image of $x_i$ under $\alpha$.
A graph is called {\em distance-transitive} if it is $k$-distance-transitive for all $k\in\nat$.
Macpherson \cite{Ma} classified the connected locally finite distance-transitive graphs.
They are exactly the graphs $X_{k,l}$, the infinite graphs of connectivity $1$ such that each block is a complete graph on $k$ vertices and every vertex lies in $l$ distinct blocks.
Here, $k$ and $l$ are integers, but we shall use the notation of $X_{\kappa,\lambda}$ also when $\kappa$ or $\lambda$ are  infinite cardinals.

Answering a question of Thomassen and Woess \cite{TW}, M\"oller \cite{Mo} showed that the 2-distance-transitive locally finite connected graphs with more than one end are still only the graphs~$X_{k,l}$.

For graphs that are not locally finite, little is known.
Our first main result is the following common generalization of the theorems of Macpherson and M\"oller to arbitrary graphs with more than one end:

\begin{Tm}\label{mainTm}
Let $G$ be a connected infinite graph with more than one end. The following properties are equivalent:
\begin{enumerate}[(i)]
\item \label{mainTmFirst}$G$ is distance-transitive;
\item \label{mainTmSecond}$G$ is $2$-distance-transitive;
\item \label{mainTmThird}$G\isom X_{\kappa,\lambda}$ for some cardinals $\kappa$ and $\lambda$ with $\kappa,\lambda\geq 2$.
\end{enumerate}
\end{Tm}

A graph is called {\em $n$-transitive} if it has no cycle of length at most $n$ and for every two paths $x_0\ldots x_m$ and $y_0\ldots y_m$ with $0\leq m\leq n$ it admits an automorphism $\alpha$ with $x_i^\alpha=y_i$ for all~$i$.

Thomassen and Woess \cite{TW} characterized the locally finite connected graphs with more than one end that are $2$-transitive.
These are precisely the $r$-regular trees for some $r\in\nat$.
As a direct consequence of Theorem \ref{mainTm} we get the following characterization of all such graphs, no necessarily locally finite:

\begin{Cor}\label{CorOfMainTm}
If $G$ is a connected $2$-transitive graph with more than one end, then $G$ is a $\lambda$-regular tree for some cardinal $\lambda\geq 2$.\qed
\end{Cor}

In the second part of this paper we investigate graphs with the property that the existence of an isomorphism $\varphi$ between two finite induced subgraphs implies that there is an automorphism $\psi$ of the entire graph mapping one of the subgraphs to the other.
This area divides into two parts: In one part $\varphi$ has to induce $\psi$ on these subgraphs, while in the other part they may differ.
More precisely, a graph $G$ is $k$-CS-transitive if for every two {\bf c}onnected isomorphic induced {\bf s}ubgraphs of order $k$ some isomorphism between them extends to an automorphism of $G$.
On the other hand, $G$ is called \emph{$k$-CS-homogeneous} if {\em every} isomorphism between two induced connected subgraphs of order $k$ of $G$ extends to an automorphism of $G$.
A graph is \emph{CS-transitive} if it is $k$-CS-transitive for all $k\in\nat$, and \emph{CS-homogeneous} if it is $k$-CS-homogeneous for all $k\in\nat$.
Furthermore, a graph is {\em end-transitive} if its automorphism group acts transitively on the set of its ends.

Gray \cite{G} classified the connected locally finite $3$-CS-transitive graphs with more than one end and showed that these graphs are end-transitive.
He asked whether all locally finite $k$-CS-transitive graphs are end-transitive.
We give a positive answer to his question, and also show that the ends of $k$-CS-transitive graphs of arbitrary cardinality have at most two orbits under the action of the automorphism group of the graph.

Since $1$-CS-transitive graphs are the transitive graphs and $2$-CS-transitive graphs are the edge-transitive graphs, there is not much hope to classify them.
Thus we investigate $k\ge3$.
We shall give a complete classification of these $k$-CS-transitive graphs with more than one end.
This is formulated in Theorem~\ref{mainTm2}.

\mbox{}

In order to state our characterization we have to introduce some classes of graphs.
Let $H, H_1, H_2$ be finite graphs, and let $\kappa,\lambda\ge 2$ be cardinals. 
Let us construct the graph $Z_{\kappa,\lambda}(H_1,H_2)$ for $\kappa,\lambda\ge 2$ as follows. Let $T$ be an infinite tree, viewed as a bipartite graph with bipartition $A,B$, and assume that the vertices in $A$ have degree $\kappa$ and the vertices in $B$ have degree $\lambda$. We replace every vertex from $A$ by an isomorphic copy of~$H_1$ and every vertex from $B$ by an isomorphic copy of~$H_2$. We add all edges between vertices that belong to graphs that replaced adjacent vertices. The resulting graph is a $Z_{\kappa,\lambda}(H_1,H_2)$.

Let $Y_\kappa$ denote a connected graph that has two different kinds of blocks, single edges and blocks of order~$\kappa$, and in which every  vertex lies in exactly one block of each kind.

Let $X_{\kappa,\lambda}(H)$  be the graph which arises from the graphs $X_{\kappa,\lambda}$ by replacing each vertex with a copy of $H$ and adding all edges between two copies replacing adjacent vertices of $X_{\kappa,\lambda}$.

We also need some finite homogeneous\footnote{{\em ultrahomogeneous} in \cite{Ga1}} graphs.
These are graphs $G$ such that any isomorphism between two finite induced subgraphs (not necessarily connected) extends to an automorphism of $G$.
These graphs were determined by Gardiner \cite{Ga1}.
Interestingly, Ronse [13] showed that the class of finite homogeneous graphs coincides with its `CS' counterpart, the class of graphs such that for any two isomorphic induced subgraphs, not necessarily connected, there {\em exists} an isomorphism between them that extends to an automorphism of the whole graph.

Another class of graphs featuring in our characterization will be a class, denoted as $\EF_{k,m,n}$, that occurs in \cite{E}.
It consists of all finite homogeneous graphs with the property that every vertex has at most $m$ neighbours, every subgraph of order at least $n$ is connected, and no two non-adjacent vertices have $k-2$ or more common neighbours.
Furthermore we exclude the complete graphs and the complements of complete graphs from $\EF_{k,m,n}$ for technical reasons.

Now we are able to state our second main result, the classification (for $k\ge 3$) of all k-CS-transitive graphs with more than one end.

\begin{Tm}\label{mainTm2}
Let $k\ge 3$. A connected graph with more than one end is $k$-CS-transitive if and only if it is isomorphic to one of the following graphs:
\begin{enumerate}[(1)]
\item\label{mainTm2First} $X_{\kappa,\lambda}(K^1)$ with arbitrary $\kappa$ and $\lambda$;
\item $X_{2,\lambda}(K^n)$ with arbitrary $\lambda$ and $n\le\frac{k}{2}$;
\item $X_{\kappa,2}(\closure{K^m})$ with arbitrary $\kappa$ and $m\le\frac{k}{3}$;
\item\label{mainTm2LastX} $X_{2,2}(E)$ with $E\in\EF_{k,m,n}$, $m\le k-2$ and $n\le \frac{k-|E|}{2}+1$;
\item $Y_\kappa$ with arbitrary $\kappa$ (if $k$ is odd);
\item\label{mainTm2FirstZ} $Z_{2,2}(\overline{K^m},K^n)$ with $2m+n\leq k+1$ (if $k$ is even);
\item $Z_{\kappa,\lambda}(K^1,K^n)$ with $n< k$, arbitrary $\kappa,\lambda$ with $\kappa=2$ or $\lambda=2$ (if $k$ is even);
\item\label{mainTm2Last} $Z_{2,2}(K^1,E)$ with $E\in\EF_{k,m,n}$, $m\leq k-2$, $n\leq\frac{k}{2}+1$ (if $k$ is even).
\end{enumerate}
\end{Tm}

Gray \cite{G} characterized the locally finite $3$-CS-homogeneous graphs with more than one end.
Since neither $Y_\kappa$ for $\kappa\ge 3$, nor $Z_{\kappa,\lambda}(H_1,H_2)$ for any distinct graphs $H_1$, $H_2$ or distinct cardinals $\kappa, \lambda$ are $k$-homogeneous, but any $k$-CS-transitive $X_{\kappa,\lambda}(H)$ is $k$-homogeneous for homogeneous finite graphs $H$, Theorem~\ref{mainTm2} allows us to extend Gray's theorem to arbitrary $k$ and graphs with arbitrary degrees, as follows:

\begin{Cor}\label{mainTm2Cor1}
Let $k\ge 3$. A connected graph with more than one end is $k$-CS-homogeneous if and only if it is isomorphic to $X_{\kappa,\lambda}(H)$ for one of the following values of $\kappa,\lambda$ and $H$:
\begin{enumerate}[(1)]
\item arbitrary $\kappa$ and $\lambda$ and $H=K^1$;
\item arbitrary $\kappa$, $\lambda=2$, $n\le\frac{k}{2}$ and $H=K^n$;
\item $\kappa=2$, arbitrary $\lambda$, $m\le\frac{k}{3}$ and $H=\overline{K^m}$;
\item $\kappa=2=\lambda$, $H\in\EF_{k,m,n}$ for $m\le k-2$ and $n\le\frac{k-|E|}{2}+1$.\qed
\end{enumerate}
\end{Cor}

Gray and Macpherson \cite{GM} classified the countable CS-homogeneous graphs, as those described in our Theorem~\ref{mainTm}.
As a further corollary of Theorem \ref{mainTm2} we can extend their classification to arbitrary graphs with more than one end.

\begin{Cor}\label{mainTm2Cor3}
For connected graphs with at least two ends the notions of being distance-transitive, CS-transitive, or CS-homogeneous coincide.
(These graphs are described in Theorem~\ref{mainTm}.)
\qed
\end{Cor}

Let us say a word about the techniques we use for our proofs.
The proofs of the corresponding theorems for locally-finite graphs are all based on Dunwoody's \emph{structure trees} corresponding to finite edge-cuts that are invariant under the action of the automorphism group of the graph.
This structure tree theory is described in the book of Dicks and Dunwoody \cite{DD}; see M\"oller \cite{Mo3,Mo2} and Thomassen and Woess \cite{TW} for introductions.
Since those edge-cuts must be finite, these structure trees can in general only be applied to locally finite graphs.

Recently Dunwoody and Kr\"on \cite{DK} developed a similar structure tree theory based on vertex cuts, providing a similarly powerful tool for the investigation of graphs that are not locally finite.
We use this new theory in our proofs.

\section{The structure tree}\label{StructureTree}
Throughout this paper we use the terms and notation from \cite{D} if not stated otherwise.
In particular, a {\em ray} is a one-way infinite path. Two rays in a graph $G$ are {\em equivalent} if there is no finite vertex set $S$ in $G$ such that the two rays lie eventually in distinct components of $G-S$. The equivalence of rays is an equivalence relation whose classes are the {\em ends} of $G$.

Let $G$ be a connected graph and $A,B\sub V(G)$ two vertex sets. The pair $(A,B)$ is called a \emph{separation} (of $G$) if $A\cup B= V(G)$ and $E(G[A])\cup E(G[B])=E(G)$.
The \emph{order} of a separation $(A,B)$ is the order of its {\em separator} $A\cap B$ and the subgraphs $G[A\sm B]$ and $G[B\sm A]$ are the \emph{wings} of $(A,B)$.
With $(A,\sim)$ we refer to the separation $(A,(V(G)\sm A) \cup N(V(G)\sm A))$, and $(\sim,A)$ respectively.
A {\em cut} is a separation $(A,B)$ of finte order with non-empty wings such that the wing $G[A\sm B]$ is connected and such that no proper subset of~$A\cap B$ separates the wings of~$(A,B)$.
A {\em cut system} $\SF$ is a non-empty set of cuts $(A,B)$ of~$G$ satisfying the following properties.
\begin{enumerate}[1.]
\item If $(A,B)\in\SF$ then there is an $(X,Y)\in \SF$ with $X\sub B$.
\item Let $(A,B)\in\SF$ and $C$ be a component of $G[B\sm A]$. If there is a separation $(X,Y)\in\SF$ with $X\sm Y\sub C$, then the separation $(C\cup N(C),\sim)$ is also in $\SF$.
\item If $(A,B)\in\SF$ with wings $X,Y$ and $(A',B')\in\SF$ with wings $X',Y'$ then there are components $C$ in $X\cap X'$ and $D$ in $Y\cap Y'$ or components $C$ in $Y\cap X'$ and $D$ in $X\cap Y'$ that are wings for separations in $\SF$ each.
\end{enumerate}
An {\em $\SF$-separator} is a vertex set $S$ that is a separator of some separation in $\SF$.

Two separations $(A_0,A_1),(B_0,B_1)\in \SF$ are \emph{nested} if there are $i,j\in \set{0,1}$ such that one wing of $(A_i\cap B_j,\sim)$ does not contain any component $C$ with ${(C \cup N(C),\sim)\in\SF}$ and $A_{i'}\cap B_{j'}$ with $i\neq i'\in\set{0,1}$ and $j\neq j'\in\set{0,1}$ contains $(A_0\cap A_1)\cup(B_0\cap B_1)$.

A cut in a cut system $\SF$ is {\em minimal} if no other cut in $\SF$ has smaller order.
A {\em minimal cut system} is a cut system all whose cuts are minimal and thus have the same order.

In our proofs we use a certain kind of minimal cut systems that was introduced by Dunwoody and Kr\"on~\cite[Example 2.2]{DK}.

\begin{Exam}\label{Example}
Let $G$ be a connected infinite graph with at least two ends.
Let $n$ be the smallest order of a (finite) vertex set $X$ such that there are at least two components in $G-X$ that contain a ray each.
Let $\SF$ be the set of all cuts $(A,B)$ with order $n$ such that both $G[A]$ and $G[B]$ contain a ray.
Then $\SF$ is a cut system.
\end{Exam}

Let $\WF$ be the set of separators $A\cap B$ with $(A,B)\in\SF$.
An \emph{($\SF$-)block} is a maximal induced subgraph $X$ such that
\begin{enumerate}[(i)]
\item for every $(A,B)\in\SF$ there is $V(X)\sub A$ or $V(X)\sub B$ but not both;
\item there is some $(A,B)\in\SF$ with $V(X)\subseteq A$ and $A\cap B\subseteq V(X)$.
\end{enumerate}
Let $\BF$ be the set of blocks.
For a nested minimal cut system $\SF$ let $\TF$ be the graph with vertex set $\WF\cup\BF$ and edges $WB$ ($W\in\WF$ and $B\in\BF$) if and only if $W\sub B$.
Then $\TF=\TF(\SF)$ is called the \emph{structure tree} of~$G$ and $\SF$.

It is the same structure tree that is used by Dunwoody and Kr\"on \cite{DK} but we use a different notation for the underlying cut system.
In particular, we describe a cut as a separation $(A,B)$ while Dunwoody and Kr\"on refer to the (connected) vertex set $A\sm B$ as a cut. 
They substantiate the term `structure tree' in~\cite[Lemma 6.2]{DK}:

\begin{Lem}
Let $G$ be a connected graph, and let $\SF$ be a nested minimal cut system.
Then the structure tree of $G$ and $\SF$ is a tree.\qed
\end{Lem}

Two vertices, vertex sets or subgraphs are {\em separated properly} by a separation (or its separator) if they lie in distinct wings of the separation and if each component in which they lie is adjacent to all vertices of the separator.

We need a fundamental property of cut systems that is shown in~\cite[Theorem~7.2]{DK} by Dunwoody and Kr\"on.
Since we do not use the whole theorem, we only state the part that is applied in this paper.

\begin{Tm}\label{VertexCuts}
Let $G$ be a connected graph with cut system $\CF$.
There is a nested cut system $\SF\sub \CF$ consisting only of minimal cuts that is invariant under $\Aut(G)$ with the following properties.
If two rays are separated by a minimal cut in~$\CF$, then they are separated by a cut in~$\SF$.
If two $\CF$-blocks are separated by a minimal cut in~$\CF$, then they are separated by a cut in~$\SF$.
Additionally, each $\CF$-block belongs to a unique vertex of the structure tree of~$G$ and $\SF$.
\qed
\end{Tm}

The structure tree $\TF$ of a connected graph $G$ and a nested minimal cut system $\SF$ is \emph{basic} if $\SF$ is an $\Aut(G)$-invariant cut system such that all separators $A\cap B$ with $(A,B)\in \SF$ belong to the same $\Aut(G)$-orbit and both wings of each separation of $\SF$ contain a ray. Furthermore we require that any cut of order less than $|A\cap B|$ does not separate two ends. 
With Example~\ref{Example} we may state a useful corollary of Theorem~\ref{VertexCuts}.

\begin{Cor}\label{basicTreeExists}
For any graph $G$ with at least two ends there is a minimal cut system $\SF$ such that the structure tree of $G$ and $\SF$ is basic.\qed
\end{Cor}

The following lemma is due to Dunwoody and Kr\"on \cite[Lemma 4.1]{DK}. We state it here as it nicely shortens some proofs.

\begin{Lem}\label{FewCuts}
For any $k$, every pair of vertices in a connected graph is separated properly by only finitely many distinct separators of order $k$.\qed
\end{Lem}

\section{Distance-transitive graphs}

Let us prove Theorem~\ref{mainTm}. Since the graphs $X_{\kappa,\lambda}$ are distance transitive, it suffices to prove that every connected $2$-distance transitive graph with at least two ends is some $X_{\kappa,\lambda}$ (with $\kappa,\lambda\ge 2$).

\begin{proof}[Proof of Theorem \ref{mainTm}]
Let $G$ be a connected $2$-distance-transitive graph with more than one end.
Let $\SF$ be a minimal cut system of $G$ such that the structure tree of $G$ and $\SF$ is basic.
Then $\SF$ is a nested cut system---in particular for every separation $(A,B)\in\SF$ and every automorphism $\alpha$ of~$G$, the cuts $(A,B),(A^\alpha,B^\alpha)$ are nested---and both wings of any  cut in $\SF$ contain a ray. 

\begin{Claim}\label{BlockComplete}
All $\SF$-blocks are complete.
\end{Claim}

Let $(A,B)\in \SF$, and let $a\in A\sm B$ and $b\in B\sm A$ be vertices with distance $2$ in~$G$.
Suppose that some $\SF$-block $X$ is not complete. Let $x,y$ be two non-adjacent vertices in $X$, and let $P$ be a shortest $x$--$y$ path in $G$.
Let $\TF'$ be \emph{the} minimal subtree of~$\TF$ containing $X$ such that its vertices cover $P$. All leaves of $\TF'$ are \SF-blocks.
Let $Y$ be a leaf of $\TF'$ different from $X$, if available, and $Y=X$ otherwise.
Then there are two vertices $a',b'\in Y$ with distance $2$ in~$G$.
As $G$ is $2$-distance-transitive there is an automorphism $\alpha$ of $G$ with $a^\alpha=a'$ and $b^\alpha=b'$.
Then $Y$ meets both wings of $(A^\alpha,B^\alpha)$---namely $a^\alpha, b^\alpha\in Y$---which contradicts the fact that $Y$ is an $\SF$-block.\qed

\begin{Claim}\label{disjointorequal}
Any two \SF-separators are equal or disjoint.
\end{Claim}

Let $S,S'$ be any two distinct \SF-separators, $(A,B)\in\SF$, and $\alpha\in\Aut(G)$ such that $S= A\cap B$ and $S^\alpha= S'$.
As $(A,B)$ and $(A^\alpha,B^\alpha)$ are nested we may assume (by symmetry) that $A\cap A^\alpha \sub B\cap B^\alpha$, and that there is a vertex $v\in A\cap A^\alpha$---$S\cap S'=\emptyset$ otherwise. As $(A,B)$ is a cut of minimal order and all blocks are complete there are vertices $x\in A^\alpha\sm S^\alpha$ and $y\in A\sm S$ with distance two. There is a vertex $x'\in S^\alpha\sm S$, since $S\ne S^\alpha$ and both separators have the same (finite) order. As $yx'\notin E(G)$ and $x,y'$ are adjacent to $v$ at a time, they have distance $2$. Thus there is an automorphism $\beta$ of~$G$ with $x^\beta=x'$ and $y^\beta= y$. This is a contradiction according to Lemma \ref{FewCuts} as there are more cuts in $\SF$---all of the same size---separating $x$ from $y$ than $x'$ from $y$.\qed

\mbox{}

Let us show that all $\SF$-separators have order 1. Suppose not, then there are at least to vertices in some $\SF$-separator $S$ and, as all \SF-blocks are complete, there is an edge $e$ in $G[S]$. On the other hand, there is an edge $e'$ that has precisely one of its end vertices in $S$.
Since $G$ is $2$-distance-transitive it is $1$-distance-transitive and thus there is an automorphism of~$G$ that maps $e$ to $e'$. This is a contradiction, since the \SF-separators intersect trivially.

As $G$ is $1$-distance-transitive any two blocks have the same order and every vertex lies in the same number of blocks. The size of an \SF-block is at least $2$, since there are edges in $G$. Every \SF-separator lies in at least two different \SF-blocks, as there are at least two different ends in~$G$. Thus $G$ is isomorphic to $X_{\kappa,\lambda}$ for some cardinals $\kappa,\lambda\ge 2$.
\end{proof}

\begin{proof}[Proof of Corollary~\ref{CorOfMainTm}]
A 2-transitive graph with at least two ends is  also $2$-distance-transitive and hence an $X_{\kappa,\lambda}$ with $\kappa,\lambda\ge 2$. If $\kappa\geq 3$ there is a path of length $2$ in every block whose (adjacent) endvertices can be mapped onto vertices with distance $2$. This is a contradiction and hence $\kappa=2$. The graphs $X_{2,\lambda}$ with $\lambda\ge 2$ are precisely the $\lambda$-regular trees.
\end{proof}

\section{The local structure for some finite subgraphs}\label{EnoSec}

In some $k$-CS-transitive graphs the previously introduced finite homogeneous graphs play a role as building blocks.
Enomoto \cite{E} gave a combinatorially characterization of these homogeneous graphs.
We apply a corollary of his result \cite[Theorem 1]{E} in our proofs.

For a subgraph $X$ of a graph $G$ let $\Gamma(X)=\bigcap_{x\in V(X)} N(x)$, which is the set of all vertices in $G$ that are adjacent to all the vertices in $X$.
A graph $G$ is {\em combinatorially homogeneous} if $\abs{\Gamma(X)} = \abs{\Gamma(X')}$ for any two isomorphic subgraphs $X$ and $X'$.
Furthermore, a graph $G$ is {\em $l$-S-transitive} if for every two isomorphic subgraphs of order $l$ there is an automorphism of $G$ mapping one onto the other.

\begin{Tm}[\mbox{\cite[Theorem 1]{E}}]\label{Eno}
Let $G$ be a finite graph.
The following properties of~$G$ are equivalent.
\begin{enumerate}[(1)]
\item $G$ is homogeneous;
\item $G$ is combinatorially homogeneous;
\item $G$ is isomorphic to
\begin{enumerate}[(a)]
\item a disjoint union of isomorphic complete graphs,
\item a complete $t$-partite graph $K^t_r$ with $r$ vertices in each partition class and with $2\le t,r$,
\item $C_5$, or
\item $L(K_{3,3})$ (the line graph of $K_{3,3}$).\qed
\end{enumerate}
\end{enumerate}
\end{Tm}

Whenever we need finite homogeneous graphs as building blocks for $k$-CS-transitive graphs we use Corollary \ref{CharEno} to handle them.

\begin{Cor}\label{CharEno}
Let $k\ge 3$, $m\leq k-2$, and $n\le \frac{k}{2}$ be integers.
Let $G$ be a finite graph with maximum degree $m$ that is neither complete nor the complement of a complete graph.
If $G$ is $l$-S-transitive for all $l\leq k-1$, any induced subgraph of $G$ on $n$ vertices is connected, and two non-adjacent vertices do not have $k-2$ common neighbours, then $G$ is (combinatorially) homogeneous and isomorphic to
\begin{enumerate}[(a)]
\item $t$ disjoint $K^r$ with $2\le t$, $1\le r-1\le m$, and $tr\le n-1$,
\item $K^t_r$ with $2\le t$, $2\le r\le n-1$, and $(t-1)r\le \min\{m, k-3\}$,
\item $C_5$ with $2\le m$ and $4\le n$, or
\item $L(K_{3,3})$ with $4\le m$ and $6\le n$.
\end{enumerate}
\end{Cor}

\begin{proof}
Theorem~\ref{Eno} provides that, ignoring the boundaries, there are no other cases as (a) to (d). The specific boundaries for each case can be checked easily. For example, in case (b) the `$k-3$' in the inequality $(t-1)r\le \min\{m, k-3\}$ ensures that $K^t_r$ does not contain two non-adjacent vertices with $k-2$ common neigbours if $m=k-2=(t-1)r$. 
\end{proof}

Let \emph{$\EF_{k,m,n}$} be the class of all those graphs that satisfy the assumptions of Corollary \ref{CharEno} with the values $k,m$ and $n$.

\section{The $k$-CS-transitivity for special graphs}\label{reverseDirection}

This section is dedicated to show that any graph $G$ from Theorem~\ref{mainTm2} is indeed $k$-CS-transitive for the specific values of~$k$.
The general idea behind the following proofs is that any connected induced subgraph of~$G$ on~$k$ vertices contains an anchor like part.

To clarify, let $G$ be a graph and let $X$ be a connected induced subgraph of $G$. A subgraph $A$ of $X$ is an \emph{anchor} of~$X$ in~$G$ if for every induced subgraph $Y$ in~$G$ isomorphic to $X$ there is some isomorphism $\gamma$ from $X$ to $Y$ such that the restricted map $\gamma|_A$ extends to an automorphism of $G$ that maps $X$ to $Y$.
\begin{Rem}
If every induced connected subgraph of order $k$ of some graph $G$ contains an anchor, then $G$ is $k$-CS-transitive.
\end{Rem}

The anchors we commonly use are either induced paths of length $3$ or  smallest separators.

The \emph{building blocks} of~$X_{\kappa,\lambda}(H)$ and $Z_{\kappa,\lambda}(H_1,H_2)$ are the isomorphic copies of~$H$, $H_1$, and $H_2$ that are used for the construction of these graphs. 

\begin{Lem}\label{rdcomplete}
Let $G$ and $k$ belong to one of the classes (\ref{mainTm2First})~to~(\ref{mainTm2Last}) of Theorem~\ref{mainTm2}.
If some connected induced subgraph $X$ of~$G$ on~$k$ vertices has diameter~$1$---i.\,e.\ $X$ is complete, then it itself is an anchor.
\end{Lem}
\begin{proof}
The only graphs from Theorem~\ref{mainTm2} that may contain complete graphs on $k$ vertices are isomorphic to some $X_{2,\lambda}(K^n)$, $X_{\kappa,\lambda}(K^1)$, $Y_\kappa$, $Z_{\kappa,2}(K^1,K^m)$, or $Z_{2,\lambda}(K^1,K^m)$.
\begin{itemize}
\item In $X_{2,\lambda}(K^n)$ any complete graph on $k$ vertices consists of precisely two building blocks or precisely two building blocks without one vertex depending on the parity of $k$.
\item In $X_{\kappa,\lambda}(K^1)$ and $Y_\kappa$ any complete graph on $k$ ($\ge 3$) vertices lies completely in some $K^\kappa$.
\item In  $Z_{\kappa,2}(K^1,K^m)$ and $Z_{2,\lambda}(K^1,K^m)$ any complete graph on $k$ vertices consists of precisely two adjacent building blocks.
\end{itemize}
In all these cases every isomorphism between complete subgraphs on $k$ vertices that respects the building blocks (or any $K^k$ if $G\isom Y_\kappa$) extends to some automorphism of the whole graph.
\end{proof}

\begin{Lem}\label{rddiam2}
Let $G$ and $k$ belong to one of the classes (\ref{mainTm2First})~to~(\ref{mainTm2Last}) of Theorem~\ref{mainTm2}.
If some connected induced subgraph $X$ of~$G$ on~$k$ vertices has diameter~$2$, then it contains an anchor.
\end{Lem}
\begin{proof}
Let $X$ be a connected induced subgraph of~$G$ on~$k$ vertices with diameter~$2$.
If $G\isom Y_\kappa$ then $X$ is isomorphic to some $K^{k-1}$ with one edge attached. This edge is an anchor. Thus we may assume that $G\not\isom Y_\kappa$.

Since all building blocks are homogeneous we may assume that $X$ meets at least two building blocks.
If $X$ meets precisely two building blocks, then $G\isom Z_{2,2}(K^1,E)$ for some graph $E\in \EF_{k,m,n}$ with $m\le k-2$ and $n\le \frac{k}{2}+1$ or $G\isom X_{2,2}(E)$ for some graph $E\in \EF_{k,k-2,n}$ with $m\le2$ and $n\le \frac{k-|E|}{2}+1$, by cardinality means. In the first case there is one vertex $v$ with $k-1$ neighbours (the building block $K^1$) which is an anchor, since the maximum degree of $E$ is at most $m\le k-2$ and $X-v$ is connected.
In the second case $E\isom C_5$ or $E\isom L(K_{3,3})$, again by cardinality.
If $E\isom C_5$, then $k=10$ and $X$ itself is an anchor. If $E\isom L(K_{3,3})$, then $15\le k\le 18$. Both of the maximal subgraphs of~$X$ that lie completely in one of the building blocks are anchors, since $L(K_{3,3})$ is $4$-connected and at most three vertices (for $k=15$) of these two building blocks do not lie in $X$.

We may assume that $X$ meets at least three building blocks. Let $B$ be the building block that touches all vertices of $X$, which exists by the small diameter of $X$.
If a separator in~$X$ does not contain every vertex of~$X\cap B$, then it must contain at least all the vertices in $X\sm B$.
In all possible cases $|X\cap B|$ is smaller than $|X\sm B|$ and $X\cap B$ is indeed the unique smallest separator.
Thus for every isomorphic induced copy $Y$ of~$X$ in~$G$ precisely the vertices of $X\cap B$ are mapped to \emph{the} smallest separator $S$ in~$Y$.
We may assume that $Y$ meets three building blocks, as it contains an anchor otherwise.
Since $S$ is a smallest separator, $S=Y\cap D$ for the unique building block $D$ of~$G$ that touches all vertices of $Y$.
Since the building blocks are homogeneous and $B$ is mapped to $D$ by some automorphism of~$G$, every isomorphism from $X$ to $Y$ extends to an automorphism of~$G$. In particular $X\cap B$ is an anchor.
\end{proof}

\begin{Lem}\label{rddiam3}
Let $G$ and $k$ belong to one of the classes (\ref{mainTm2First})~to~(\ref{mainTm2Last}) of Theorem~\ref{mainTm2}.
If some connected induced subgraph $X$ of~$G$ on~$k$ vertices has diameter at least three, then it contains an anchor.
\end{Lem}
\begin{proof}
If $G\isom Y_\kappa$, then every minimal separator is a single vertex and an anchor.

For the other cases, let $X$ be some connected induced subgraph of~$G$ on~$k$ vertices with diameter at least three, and let $P$ be an (induced) path of length $3$ in~$X$ that meets four building blocks of~$G$.
Such a path exists since there is no building block $B$ such that $X\cap B$ contains an induced path of length $3$ whose end vertices have distance $3$ in $X$.

Let $\gamma: X\to Y$ be some isomorphism for an induced subgraph $Y$ of~$G$.
We further require that $v$ and $v^\gamma$ for each vertex $v\in P$ belong to building blocks in the same orbit of the automorphism group of~$G$.
This is a legitimate request, since the number of vertices in $P$ is even, and if $P$ embeds into $P^\gamma$ uniquely, then there are stars or triangles in $X$ and $Y$ that force $P$ and $P^\gamma$ to be aligned or $P$ is a path of length $k-1$ and thus it itself is an anchor.

Let us recursively construct an automorphism of~$G$ that maps $X$ to $Y$.
Let $\alpha_0$ be an automorphism of~$G$ with $\alpha_0|_{P}=\gamma$ such that vertices in the (homogeneous) building blocks containing $P$ are mapped to $Y$ if and only if they lie in $X$.

To define the automorphism $\alpha_l$ of~$G$ for $l\ge 1$ let $\alpha_i$ be defined for $i<l$.
First, let $W$ be the set of vertices in~$G$ with distance at most $l-1$ to the building blocks that contain~$P$.
The graphs $X$ and $Y$ induce graphs $X_1,\dots,X_n$ and $Y_1,\dots,Y_n$ with $X_j^\gamma=Y_j$ for all $1\le j\le n$ in the components of $G-W$ and $G-W^{\alpha_{l-1}}$, respectively.
Let $\alpha_l$ be an automorphism of $G$ with $w^{\alpha_l}:=w^{\alpha_{l-1}}$ for $w\in W$, that maps the component of~$G-W$ containing $X_j$ to the component of~$G-W^{\alpha_{l-1}}$ containing $Y_j$ for all $j\le n$ such that the vertices of $X$ adjacent to $W$ are mapped precisely to those vertices of $Y$ adjacent to $W^{\alpha_{l-1}}$.
Since the diameter of $X$ is less than $k$, the automorphism $\alpha_k$ of $G$ maps $X$ onto $Y$.
\end{proof}

These three lemmas show that in all cases there are anchors as needed.

\section{The global structure of $k$-CS-transitive graphs}

The following two lemmas can be shown for $k\le 2$ easily.
As the lemmas with $k\le 2$ do not play any role in the proof of Theorem \ref{mainTm2}, we do not provide the corresponding proofs for smaller $k$.

\begin{Lem}\label{noLeaves}
If $G$ is a connected $k$-CS-transitive graph with at least two ends and $k\ge 3$, then every basic structure tree of $G$ has no leaves.
\end{Lem}

\begin{proof}
Let $\SF$ be a minimal cut system of $G$ such that the structure tree $\TF$ of $G$ and $\SF$ is basic.
Suppose that $\TF$ contains a leaf. Let $X$ be some $\SF$-block representing a leaf in $\TF$, and let $(A,B)\in\SF$ be a separation with $X\sub A$ and $A\cap B\sub X$.
Then $A\cap B$ is the only $\SF$-separator in $X$ and $X=A$.
Since there is a ray in $G[A]$, the block $X$ is infinite.
There is no vertex in $X$ that has distance $k+1$ to $B$, as an induced path starting in $A\cap B$ could be mapped into $X\sm (A\cap B)$ and as this would contradict the fact that $A\cap B$ is the only $\SF$-separator in $X$.
Thus there are vertices of infinite degree in $X$.
Let the vertex $x\in X$ have infinite degree and minimal distance to $B$ with this property. 
Let $N$ be the infinite set of neighbours of $x$ with $d(v,B)>d(x,B)$ for all $v\in N$.
Then there is a $K^{\aleph_0}$ or its complement in $G[N]$.
As $k-2$ independent vertices in $N$ together with $x$ and one neigbour of $x$ that is neither contained in $N$ nor adjacent to any vertex in $N$ induce a subgraph that could be mapped onto a subgraph induced by $k-1$ independent vertices in $N$ and $x$, there have to be a $K^{\aleph_0}$ in $G[N]$.
This yields to a contradiction, too.
Let $H$ be a complete graph on $k$ vertices in $G[N]$, and let $v\in V(H)$.
Then there is no automorphism of $G$ that maps $H-v+x$ to $H$, which is a  contradiction to the $k$-CS-transitivity of~$G$.
\end{proof}

\begin{Lem}\label{infDiam}
For $k\ge 3$, every connected $k$-CS-transitive graph $G$ with at least two ends has infinite diameter.
\end{Lem}

\begin{proof}
Let $\SF$ be a minimal cut system of $G$ such that the structure tree $\TF$ of $G$ and $\SF$ is basic.
Then there is a double ray $R$ in $\TF$ as there is no leaf in~$\TF$ by Lemma \ref{noLeaves}.
This ray hits infinitely many different (finite) $\SF$-separators.
Suppose that there is a vertex $x$ in $G$ that lies in infinitely many of these separators.
Since $x$ has neighbours in infinitely many $\SF$-blocks on $R$, this results in two induced stars with $k-1$ leaves in~$G$ whose corresponding $\SF$-blocks (regarded as vertices in $\TF$) induce vertex sets with different diameter in $\TF$.
This is a contradiction according to Lemma~\ref{FewCuts}, since the leaves of these two stars that are furthest away in $\TF$ are separated by a different number of separators of the same (finite) order.
Thus we conclude that there are infinitely many pairwise disjoint $\SF$-separators on $R$.
Two $\SF$-separators $S_1,S_2$ that have $n$ disjoint $\SF$-separators on their $S_1$--$S_2$ path in $\TF$ have distance at least $n$ in~$G$.
\end{proof}

The separators come in very handy.
With a simple application of Menger's Theorem one may construct order of the separators many disjoint rays in~$G$ following any ray in a basic structure tree of~$G$.
Every such ray in~$G$ induces a (connected) path in every block and contains at most one vertex from each separator.
This implies that every vertex lies in some block.

\begin{Lem}\label{SepDisj}
Let $k\ge 3$, let $G$ be a connected $k$-CS-transitive graph with at least two ends and let $\SF$ be a minimal cut system of~$G$ such that the structure tree of~$G$ and $\SF$ is basic.
Let $S$ be an $\SF$-separator.
If every $s\in S$ has for every $\SF$-block $X$ containing $S$ an adjacent vertex in~$X\sm S$, then $S$ is disjoint to any other $\SF$-separator $S'$.
\end{Lem}

\begin{proof}
Suppose that distinct $\SF$-separator $S$, $S'$ contain a common vertex $s$ and for every $\SF$-block $X$ containing $S$ there is an edge between $s$ and $X\sm S$. Let $(A,B),(A',B')\in\SF$ be cuts with separators $S$ and $S'$, respectively. 
We may assume that $S'\sub B$ and $S\sub B'$ since $\SF$ is nested.

There is an induced path $P$ of length $k-2$ ending in $s$ whose other vertices lie in~$A\sm B$.
By assumption $s$ has at least two neighbours $x$ and $y$ such that $x$ lies in $B'\sm A'$ and $y$ lies $A'\sm  A$.
The paths $Px$ and $Py$ of length $k-1$ can be mapped onto each other with an automorphism of $G$ by the $k$-CS-transitivity of $G$.
But the endvertices of~$Px$ and of~$Py$ are separated properly by a different number of $\SF$-separators, since any $\SF$-separator separating the endvertices of $Py$ properly separates also the endvertices of $Px$ properly as $\SF$ is nested, and on the other hand the separator $S'$ separates only the endvertices of $Px$ properly.
This contradicts the choice of $x$ and $y$.
\end{proof}

Let $k\ge 3$, let $G$ be a connected $k$-CS-transitive graph with at least two ends, and let $\SF$ be a minimal cut system of $G$ such that the basic structure tree of~$G$ and $\SF$ is basic.
There are two profoundly different cases. In the first case the graph is covered with $\SF$-separators while in the second case there are vertices in $G$ that do not belong to any $\SF$-separator.

Before we begin investigating these cases we need some definitions.
For an $\SF$-block $X$ we define the \emph{open ($\SF$-)block}
$$\open{X}:=X\sm\bigcup\set{A\cap B\mid (A,B)\in\SF}.$$
A {\em $k$-spoon} is an induced subgraph of $G$ that consists of a triangle and a path starting in one of its triangel vertices with all in all precisely $k$ vertices.
A spoon $H$ {\em pokes} in an $\SF$-block $X$, an $\SF$-separator $S$, or two $\SF$-separators $S,S'$ if its degree 2 vertices of the triangle are contained in $\open{X}$, $S$, or one in $S$ and one in $S'$, respectively.
A {\em $k$-fork} is another induced subgraph of $G$ that consists of its {\em prongs}, a pair of two non-adjacent vertices, and of its \emph{handle}, a path such that both prongs are adjacent to the same endvertex of the handle, and has $k$ vertices.
A fork $H$ {\em pokes} in an $\SF$-block $X$, an $\SF$-separator $S$, two $\SF$-blocks $X,Y$, or two $\SF$-separators $S,S'$ if its prongs are contained in $\open{X}$, in $S$, meet $\open{X}$ and $\open{Y}$, or meet $S$ and $S'$, respectively.

\subsection{Empty open blocks}

This is the slightly simpler case.
If $k$ is odd this is the only possible case as we will show in Lemma \ref{openBlocksEven}.
 
\begin{Lem}\label{completerpartite}
Let $k\ge 3$, let $G$ be a connected $k$-CS-transitive graph with at least two ends, and let $\SF$ be a minimal cut system such that the structure tree of $G$ and $\SF$ is basic and such that for some $\SF$-block $X$ its open block $\open{X}$ is empty.
If $S,S'$ are distinct $\SF$-separators that both lie in $X$, then $ss'\in E(G)$ for all $s\in S$ and $s'\in S'\sm S$, and any two distinct $\SF$-separators in~$X$ are disjoint.
\end{Lem}

\begin{proof}
In the case that all vertices lie in $\SF$-separators, there is a path of arbitrary length (following a path in the structure tree of~$G$ and $\SF$) such that any two vertices with distance $2$ or greater do not lie in the same $\SF$-block.
There is also an induced $s$--$s'$ path $P$ whose inner vertices do not meet the component $C$ of~$G-S'$ that contains $S\sm S'$.
If the length of $P$ is less than $k-1$ we elongate $P$ from $s$ into $C$.
Thus there is an induced subpath $P'$ of $P$ of length $k-1$.
By the $k$-CS-transitivity any two vertices of distance at least two on $P'$ and hence also on $P$ lie in different blocks.
This implies $d(s,s')<2$.
With Lemma~\ref{SepDisj} it follows that any two distinct $\SF$-separators in~$X$ are disjoint.
\end{proof}

\begin{Lem}\label{Y_kappa}
Let $k\ge 3$, let $G$ be a connected $k$-CS-transitive graph with at least two ends, and let $\SF$ be a minimal cut system such that the structure tree of $G$ and $\SF$ is basic.
If every open $\SF$-block is empty, then any two $\SF$-blocks are isomorphic, or $k$ is odd and there is a cardinal $\kappa$ such that $G\isom Y_\kappa$.
\end{Lem}

\begin{proof}
Suppose that there are two $\SF$-blocks $X$ and $Y$ that are not isomorphic.
Since there is an automorphism $\alpha$ of $G$ with $X\cap Y^\alpha\ne\es$, we may assume that $X\cap Y\ne\emptyset$.
If $X\cap Y$ contains two distinct vertices, there is either a $k$-fork with both prongs in $X$ and one $k$-fork with both prongs in $Y$, or there is a $k$-spoon with its triangle---the subgraph isomorphic to a $K^3$---in $X$ and one $k$-spoon with its triangle in $Y$.
This contradicts the $k$-CS-transitivity of $G$.
Thus the $\SF$-blocks intersect in at most one vertex, and every $\SF$-block is complete by Lemma \ref{completerpartite}.
If there are two $\SF$-blocks with more than $2$ vertices each, then for both these $\SF$-blocks there is a $k$-spoon poking it, which implies that they are $\Aut(G)$-isomorphic.
As every $\SF$-block contains an edge we know that there are precisely two different kinds of $\SF$-blocks:
one $\SF$-block is isomorphic to a $K^2$ and another one is isomorphic to a $K^\kappa$ for some $\kappa\ge 3$.
This yields to the $Y_\kappa$.
No $Y_\kappa$ with $\kappa\geq 3$ is $k$-CS-transitive for even $k$ since there is a path of length $k-1$ with both outermost edges in $\SF$-blocks isomorphic to a $K^2$ and there is a path of length $k-1$ with both outermost edges in $\SF$-blocks isomorphic to a $K^\kappa$ with $\kappa\ge 3$.
There is no automorphism of $G$ mapping the first onto the second path.
This completes the proof.
\end{proof}

\begin{Lem}\label{isomorphicBlocks}
Let $k\ge 3$, let $G$ be a connected $k$-CS-transitive graph with at least two ends, and let $\SF$ be a minimal cut system such that the structure tree of $G$ and $\SF$ is basic and every open $\SF$-block is empty.
If any two $\SF$-blocks are isomorphic, then there are precisely two $\SF$-separators per $\SF$-block and every $\SF$-separator lies in precisely two $\SF$-blocks, or there are cardinals $\kappa,\lambda\ge 3$ and integers $2\le m\le\frac{k}{3}$ and $2\le n\le\frac{k}{2}$ such that $G\isom X_{2,\lambda}(\closure{K^m})$ or $G\isom X_{\kappa, 2}(K^n)$ or $G\isom X_{\kappa,\lambda}(K^1)$.
\end{Lem}

\begin{proof}
We already know that $G\isom X_{\kappa,\lambda}(H)$ for some finite graph $H$ and we may assume that $\kappa>2$ or $\lambda>2$.
If there are edges in $H$ and $\lambda\ge 3$ then there are two kinds of $k$-spoons: one with its triangle meeting $3$ $\SF$-separators and one meeting precisely two $\SF$-separators.
If there are two non-adjacent vertices in $H$ and $\kappa\ge 3$ then there are two kinds of $k$-forks: one pokes in a single separator and one pokes in two different separators.
Thus there is $G\isom X_{2,\lambda}(\closure{K^m})$, $G\isom X_{\kappa, 2}(K^n)$, or $G\isom X_{\kappa,\lambda}(K^1)$ with $m,n\ge 2$.
It remains to show that $m\le\frac{k}{3}$ and $n\le\frac{k}{2}$.
Suppose that $m\ge \frac{k}{3}+1$ and let $S_1,S_2$ be $\SF$-separators in different $\SF$-blocks both adjacent to some $\SF$-separator $S_0$.
Let $A_i\sub S_i$ for $i=0,1,2$ with $|A_0|=|A_1|=\frac{k}{3}+1$ and $|A_2|=\frac{k}{3}-2$.
Let $B_i\sub S_i$ for $i=0,1,2$ with $|B_i|=\frac{k}{3}+1-i$.
Then there is no automorphism of $G$ from $G[\bigcup A_i]$ to $G[\bigcup B_i]$ althogh both induced subgraphs are isomorphic.
Thus there is $m\le\frac{k}{3}$.

Suppose $n\ge\frac{k}{2}+1$.
Let $S_0,S_1$ be two adjacent $\SF$-separators.
Let $A_i\sub S_i$ with $|A_0|=\frac{k}{2}+1$ and $|A_1|=\frac{k}{2}-1$, and let $B_i\sub S_i$ with $|B_i|=\frac{k}{2}$.
Then there is no automorphism of $G$ from the complete graph on $k$ vertices $G[\bigcup A_i]$ to the complete graph on $k$ vertices $G[\bigcup B_i]$.
Thus $n\le\frac{k}{2}$ follows.
\end{proof}

\begin{Lem}\label{twoBlocks}
Let $k\ge 3$, let $G$ be a connected $k$-CS-transitive graph with at least two ends, and let $\SF$ be a minimal cut system such that the structure tree of $G$ and $\SF$ is basic, every open $\SF$-block is empty and all $\SF$-blocks are isomorphic.
If every $\SF$-block contains precisely two $\SF$-separators and every $\SF$-separator is contained in precisely two $\SF$-blocks, then $G\isom X_{2,2}(E)$ with $E\in\EF_{k,m,n}$, $m\le k-2$, and $n\ge \frac{k-|E|}{2}+1$, $E\isom\closure{K^m}$ with $1\leq m\le\frac{k}{3}$, or $E\isom K^n$ with $1\le n\le\frac{k}{2}$.
\end{Lem}

\begin{proof}
As the open $\SF$-blocks are empty it is obvious that $G\isom X_{2,2}(E)$ for some finite graph $E$.
If $E$ is a complete graph or the complement of a complete graph, then the same proof as in Lemma \ref{isomorphicBlocks} shows that $1\le m\le\frac{k}{3}$ if $E\isom \closure{K^m}$ or $1\leq n\le\frac{k}{3}$ if $E\isom K^n$.
By Theorem \ref{CharEno} it suffices to show that (a) $E$ is $l$-S-transitive for all $l\le k-1$, (b) $\Delta(E)\le k-2$, (c) any vertex set of order $\frac{k-|E|}{2}+1$ in $E$ is connected, and (d) no two non-adjacent vertices of $E$ have $k-2$ common neighbours.
\begin{enumerate}[(a)]
\item Let $A, B\sub S$ be isomorphic graphs with at most $k-1$ vertices for some $\SF$-separator $S$ ($\isom E$).
Then there is a common neighbour $v$ of these vertices in an adjacent $\SF$-separator.
Adding a path of suitable length that starts in $v$ and that is only in $v$ adjacent to $S$, one gets two connected subgraphs of order $k$ and an automorphism of $G$ mapping one to the other.
If $\abs{A}\neq 1$, then this automorphism must map the vertices of $A$ onto vertices of $S$ and hence onto $B$.
If $\abs{A}=1$, let $S'$ be an $\SF$-separator such that some induced path of length $k-1$ starting in $A$ ends in $S'$.
Let $\varphi,\varphi'$ be the isomorphisms from $E$ to $S,S'$, respectively.
Let $A'\sub S'$ be $(A^{\varphi\inv})^{\varphi'}$.
Then we may assume that the path ends in $A'$.
Thus there has to be an automorphism of $G$ mapping $A$ to $B$ or $A'$ to $B$.
\item Let $S$ ($\isom E$) be an $\SF$-separator.
Suppose there is a vertex $v$ of degree at least $k-1$ in $G[S]$.
Let $A\sub S$ contain $v$ and $k-1$ of its neigbours.
Let $w$ be some vertex from an $\SF$-separator that is adjacent to $S$.
Then $G[A-v+w]$ is isomorphic to $G[A]$ but there is no automorphism of $G$ mapping the one onto the other.
Thus no vertex in $S$ has degree $k-1$ or greater.
\item Finally, suppose there is a vertex set $X\sub V(E)$ of order at least $\frac{k-|E|}{2}+1$ that is not connected.
Let $S,S',S''$ be three distinct $\SF$-separators such that $S'$ is adjacent to the other two.
Let $A\sub S,A''\sub S''$ be copies of $X$ in $S$ and $S''$ that contain the components $C,D\sub A$ and $C'',D''\sub A''$, respectively.
Let $c\in C$ and $c''\in C''$ be two vertices such that $C-c$ and $C''-c''$ are isomorphic.
Let $d\in D$ be any vertex.
Then the graphs $G[A\cup S'\cup A'']-\{c,d\}$ and $G[A\cup S'\cup A'']-\{c'',d\}$ are isomorphic but there is no $G$-automorphism mapping one to the other.
There are $k$ or $k+1$ vertices in these subgraphs, depending on the parity of $\abs{E}$.
Since the argument stays valid even if we ignore one of the vertices in $S'$  there is such a CS-transitivity contradicting graph with precisely $k$ vertices.
\item Suppose that there are two non-adjacent vertices $x,y$ in some $\SF$-separator $S'$~($\isom E$) with $k-2$ common neighbours and let $N\sub S'$ be $k-2$ of these neighbours.
Let $S, S''$ be distinct $\SF$-separators adjacent to $S'$ and let $s\in S$ and $s''\in S''$.
Then $G[N+x+y]$ and $G[N+s+s'']$ are isomorphic but there is no automorphism of $G$ mapping one onto the other.\qedhere
\end{enumerate}
\end{proof}

By Lemma \ref{Y_kappa}, \ref{isomorphicBlocks}, and \ref{twoBlocks} we may finish the first case.

\begin{Tm}\label{emptyOpenBlocks}
Let $k\ge 3$, let $G$ be a connected $k$-CS-transitive graph with at least two ends, and let $\SF$ be a minimal cut system of $G$ such that the structure tree of $G$ and $\SF$ is basic and every open block is empty.
Then there are cardinals $\kappa,\lambda\ge 2$ and integers $m,n$ such that $G$ is isomorphic to one of the following graphs:

\begin{enumerate}[(1)]
\item $X_{\kappa,\lambda}(K^1)$,
\item $X_{\kappa,2}(K^n)$ with $n\le\frac{k}{2}$,
\item $X_{2,\lambda}(\closure{K^m})$ with $m\le\frac{k}{3}$,
\item $X_{2,2}(E)$ with $E\in\EF_{k,m,n}$, $m\le k-2$ and $n\le \frac{k-|E|}{2}+1$,
\item $Y_\kappa$ (if $k$ is odd).\qed
\end{enumerate}
\end{Tm}

\subsection{A non-empty open block}

Let us discuss the connected $k$-CS-transitive graphs with at least two ends for $k\ge 3$ such that no orbit of any smallest $\SF$-separator, that separates ends, covers the whole graph.
In other words, there is a non-empty open $\SF$-block.
As mentioned before this case restricts $k$ to be even.

\begin{Lem}\label{openBlocksEven}
Let $k\ge 3$ and $G$ be a connected $k$-CS-transitive graph with at least two ends and let $\SF$ be a minimal cut system such that the structure tree of $G$ and $\SF$ is basic. If some open $\SF$-block is non-empty, then $k$ is even.
\end{Lem}

\begin{proof}
Recall that there is an induced ray and a path with $k$ vertices whose middle vertex lies in some $\SF$-separator if $k$ is odd. We may map the path anywhere into the ray and thus know that there are $k$ succeeding vertices on the ray that belong to an $\SF$-separator. As the diameter is infinite in every vertex starts an induced path of length  $k-1$ an thus every vertex lies on a path all whose vertices lie in some $\SF$-separator. Thus, if $k$ is odd, then every vertex lies in some $\SF$-separator.
\end{proof}

\begin{Lem}
Let $k\ge 3$ and let $G$ be a connected $k$-CS-transitive graph with at least two ends and let $\SF$ be a minimal cut system such that the structure tree of $G$ and $\SF$ is basic and some open $\SF$-block is not empty. If vertices $s,s'\in G$ belong to different $\SF$-separators then $s$ and $s'$ are not adjacent.
\end{Lem}

\begin{proof}
Suppose $s$ and $s'$ are adjacent. Then there is some induced path $P$ of length $k-1$ in $G$ such that the innermost ($k$ is even) edge is $ss'$. Mapping this path with the edge $ss'$ to successive edges on an induced ray, we obtain a ray all whose edges have end vertices only in $\SF$-separators. 
This is a contradiction since in every vertex starts an induced path of length $k-1$ and at least one vertex lies in an open $\SF$-block.
\end{proof}

\begin{Lem}
Let $k\ge 3$ and let $G$ be a connected $k$-CS-transitive graph with at least two ends and let $\SF$ be a minimal cut system such that the structure tree of $G$ and $\SF$ is basic and some open $\SF$-block is not empty. If vertices $s,x\in G$ lie in the same $\SF$-block with $s$ in some $\SF$-separator and $x$ not, then $s$ and $x$ are adjacent.
Furthermore, any two distinct $\SF$-separators are disjoint.
\end{Lem}

\begin{proof}
There is a path of arbitrary length such that any two vertices with distance at least $3$ on the path do not lie in the same $\SF$-block and every other vertex on this path lies in an $\SF$-separator.
As there is some induced path between $s$ and $x$ which can be extended if necessary to an induced path of length $k-1$, we know that $d(x,s)<3$. Since $x$ does not lie in any $\SF$-separator and every path enters and leaves $\SF$-blocks through $\SF$-separators $d(x,s)<2$ holds.
By Lemma~\ref{SepDisj} and Lemma~\ref{completerpartite} we may conclude that distinct $\SF$-separators are disjoint. 
\end{proof}

The final step to show that these $k$-CS-transitive graphs with non-empty open blocks resemble some $Z_{\kappa,\lambda}(H_1,H_2)$ is that their automorphism group acts transitively on its open blocks.

\begin{Lem}
Let $k\ge 3$ and let $G$ be a connected $k$-CS-transitive graph with at least two ends and let $\SF$ be a minimal cut system such that the structure tree of $G$ and $\SF$ is basic and some open $\SF$-block is not empty. The automorphism group of $G$ act transitively on the open $\SF$-blocks.
\end{Lem}

\begin{proof}
Since $\TF$ has no leaves, every induced path in $G$ of length $2$ through some $\SF$-block $X$ can be elongated to an induced path $P$ of length $k-1$ such that the innermost edge ($k$ is even) of $P$ lies in no other $\SF$-block than $X$. A similar path $P'$ can be found for any other $\SF$-block $X'$ and hence there is an automorphism $\alpha$ of~$G$ with $P^\alpha=P'$ and thus also $X^\alpha=X'$.
\end{proof}

Thus every connected $k$-CS-transitive graph for $k\ge 3$ with more than one end and some non-empty open block is isomorphic to $Z_{\kappa,\lambda}(H_1,H_2)$ for some graphs $H_1$ and $H_2$.
It remains to specify the building blocks and possible values for $\kappa$ and $\lambda$ of these graphs.

\begin{Lem}\label{mn2}
Let $k\ge 3$ and let $G\isom Z_{\kappa,\lambda}(H_1,H_2)$ be a $k$-CS-transitive graph and let $\SF$ be a minimal cut system such that the structure tree of $G$ and $\SF$ is basic and some open $\SF$-block is not empty.
The following holds:
\begin{enumerate}[(i)]
\item\label{mn2Part1} At least one of $\kappa$ or $\lambda$ is $2$.
\item\label{2ndINmn2}\label{mn2Part2} If $H_i$ contains two non-adjacent vertices, then $H_j$ ($j\neq i$) is complete and $\kappa=\lambda=2$.
\item\label{mn2Part3} If $H_i$ contains an edge, then $H_j$ ($i\neq j$) contains no edge.
\end{enumerate}
\end{Lem}

\begin{proof} (Recall that $H_1\not\isom H_2$ or $\kappa\ne\lambda$ since the copies of $H_1$ and $H_2$ are not $\Aut(G)$-isomorphic.)
Suppose $\kappa,\lambda\ne 2$, then there are two $k$-forks. One that pokes in two distinct open $\SF$-blocks, and one that pokes in two distinct $\SF$-separators. But there is no automorphism of~$G$ mapping one to the other. This proves (i).

With an analog argument follows (\ref{mn2Part2}): Suppose $\kappa$ or $\lambda$ is greater than $2$. Then there is a $k$-fork that pokes just one copy of an $H_i$ and one that pokes two distinct $\SF$-separators ($\lambda>2$) or two distinct open $\SF$-blocks ($\kappa>2$). Suppose on the other hand that there are two non-adjacent vertices in $H_j$, then there are two incompatible $k$-forks, too. One pokes an open $\SF$-block and the other one an $\SF$-separator.

For (iii), suppose that $H_i$ as well as $H_j$ contain edges. Then there are $k$-spoons  that poke an open $\SF$-block and others that poke an $\SF$-separator.
\end{proof}

From the previous lemma we immediately get the following corollary:

\begin{Cor}\label{complete-complement}
Let $k\ge 3$ and let $G\isom Z_{\kappa,\lambda}(H_1,H_2)$ be a $k$-CS-transitive graph and let $\SF$ be a minimal cut system such that the structure tree of $G$ and $\SF$ is basic and some open $\SF$-block is not empty. If both $H_1$ and $H_2$ have at least two vertices, one is a complete graph and the other one is the complement of a complete graph and $\kappa=\lambda=2$.\qed
\end{Cor}

To finish the proof in the situation that both $H_1$ and $H_2$ have at least two vertices, we will restrict the order of those graphs:

\begin{Lem}
Let $k\ge 3$ and let $G\isom Z_{2,2}(H_1,H_2)$ be a $k$-CS-transitive graph and let $\SF$ be a minimal cut system such that the structure tree of $G$ and $\SF$ is basic and some open $\SF$-block is not empty.
If $H_1$ contains two non adjacent vertices or $H_2$ contains an edge, then $H_1\isom \closure{K^m}$ and $H_2\isom K^n$ with $2m+n\leq k+1$.
\end{Lem}

\begin{proof}
By Corollary \ref{complete-complement} it suffices to show that the boundaries for $m$ and $n$ hold.
Let $H_1$ be the complement of a complete graph and let $H_2$ be a complete graph. Let $S$ be some $\SF$-separator, and let $X$ and $Y$ be two distinct $\SF$-blocks with $S\sub X,Y$.
Then any graph with precisely $k$ vertices that consists of $S$, more than $\frac{k-n}{2}$ ($+\frac{1}{2}$ if $n$ is odd) vertices in $X$ and less than $\frac{k-n}{2}$ vertices in $Y$ can be mapped onto a graph consisting of $S$, $\frac{k-n}{2}$ ($+\frac{1}{2}$ if $n$ is odd) many vertices in $X$ and $\frac{k-n}{2}$ ($-\frac{1}{2}$ if $n$ is odd) many vertices in $Y$.
Thus
$$2m+n \le 1+n+(\frac{k-n}{2}+\frac{1}{2})+(\frac{k-n}{2}-\frac{1}{2})\le k+1.$$
We seemingly loose one vertex in the case that $n$ is even. Since $2m+n$ and $k$ are even then, this is not a true loss.
\end{proof}

\begin{Lem}
Let $k\ge 3$ and let $G\isom Z_{\kappa,\lambda}(H_1,H_2)$ be a $k$-CS-transitive graph and let $\SF$ be a minimal cut system such that the structure tree of $G$ and $\SF$ is basic and some open $\SF$-block is not empty. If $\abs{H_1}=1$ and one of $\kappa$ and $\lambda$ is not $2$, $H_2$ is a complete graph on at most $k-1$ vertices.
\end{Lem}

\begin{proof}
The first part follows directly from Lemma \ref{mn2}\,(\ref{2ndINmn2}).
For the second part, suppose that $H_2$ has more than $k-1$ vertices.
Then every open $\SF$-block $\open{X}$ contains an isomorphic copy of a $K^k$ and there is a second isomorphic copy of a $K^k$ with $k-1$ vertices in $\open{X}$ and one vertex in some $\SF$-separator $S\sub X$. Since there is no automorphism of~$G$ mapping the one onto the other, $H_2$ has at most $k-1$ vertices.
\end{proof}

As the last part in this case of the proof (that there is some non-empty open block) we will determine the graphs $H_2$ if $H_1$ is only one vertex and the open blocks are neither complete nor complements of complete graphs.

\begin{Lem}
Let $k\ge 3$, let $G\isom Z_{2,2}(H_1,H_2)$ be a $k$-CS-transitive graph with at least two ends, and let $\SF$ be a minimal cut system of $G$ such that the structure tree of~$G$ and $\SF$ is basic and that some open $\SF$-block is not empty. If $\abs{H_1}=1$ and $H_2$ is neither complete nor a complement of a complete graph, then $H_2\in \EF_{k,m,n}$ with $m\leq k-2$ and $n\leq\frac{k}{2}+1$.
\end{Lem}

\begin{proof}
By Corollary \ref{CharEno} it suffices to show that (a) $H_2$ is $l$-S-transitive for all $l\le k-1$, (b) $\Delta(H_2)\le k-2$, (c) any vertex set of order $\frac{k}{2}+1$ in $H_2$ is connected, and (d) no two non-adjacent vertices of $E$ have $k-2$ common neighbours.

The proofs of~(a), (b) and~(d) are analog to those of Lemma~\ref{twoBlocks}~(a), (b) and~(d).

\begin{enumerate}[(a)]
\setcounter{enumi}{2}
\item We follow the argument of Lemma~\ref{twoBlocks} (c).
The boundary of $n$ is slightly different, since there is only one vertex in an $\SF$-separator available.
Thus every set of order $\frac{k-1}{2}+1$ is connected, since $k$ is even a set with at least $\frac{k-1}{2}+1$ vertices has at least $\frac{k}{2}+1$ vertices.\qedhere
\end{enumerate}
\end{proof}

These lemmas let us finish the case for non-empty open blocks.

\begin{Tm}\label{openBlocks}
Let $k\ge 3$, let $G\isom Z_{2,2}(H_1,H_2)$ be a $k$-CS-transitive graph with at least two ends, and let $\SF$ be a minimal cut system of $G$ such that the structure tree of~$G$ and $\SF$ is basic and that some open $\SF$-block is not empty.
Then $k$ is even and $G$ is isomorphic to one of the following graphs:
\begin{enumerate}[(1)]
\item[(6)] $Z_{2,2}(\overline{K^m},K^n)$ with $2m+n\leq k+1$ and $m\leq n$;
\item[(7)] $Z_{\kappa,\lambda}(K^1,K^n)$ with $n\leq k-1$ and cardinals $\kappa,\lambda$ with $\kappa=2$ or $\lambda=2$;
\item[(8)] $Z_{2,2}(K^1,E)$ with $E\in\EF_{k,m,n}$, $m\leq k-2$ and $n\leq\frac{k}{2}+1$.\qed
\end{enumerate}
\end{Tm}
With Section~\ref{reverseDirection} the Theorems~\ref{emptyOpenBlocks} and \ref{openBlocks} implie one of our main results, Theorem~\ref{mainTm2}.

\section{Ends of $k$-CS-transitive graphs}\label{EndAction}

Gray \cite{G} asked whether every locally finite $k$-CS-transitive graph (with $k\ge 3$) is end-transitive.
With Theorem~\ref{mainTm2} we may answer his question.

\begin{Tm}
Let $k\ge 3$ and let $G$ be a connected locally finite graph. If $G$ is $k$-CS-transitive, then it is end-transitive.\qed
\end{Tm}
This theorem does not extend to graphs with vertices of infinite degree.
For example the graphs $X_{\kappa,\lambda}$ with $\kappa\ge\aleph_0, \lambda\ge 2$ contain fundamentally different ends.
Let us make this precise, there are \emph{local} ends containing only rays that meet a vertex set  of finite diameter again and again, and there are \emph{global} ends that do not contain any such ray.
Theorem~\ref{mainTm2} shows that in $k$-CS-transitive graphs with $k\ge 3$ every end is either local or global.

\begin{Tm}
Let $k\ge 3$ and $G$ be a connected $k$-CS-transitive graph with more than one end. Then every end of G is either local or global. The automorphism group of~$G$ acts transitively on the local ends, as well as on the global ends.
$G$ is end-transitive if and only if it has no local end.
 \qed
\end{Tm}

Kr\"on and M\"oller \cite{K,KM} introduced metric ends.
A {\em metric ray} in some graph $G$ is a ray in $G$ such that any infinite subset of its vertices has infinite diameter.
Two metric rays $R_1$ and $R_2$ are {\em metrically equivalent} if there is no vertex set $S$ of finite diameter such that $R_1$ and $R_2$ lie eventually in different components of $G-S$.
A {\em metric end} is an equivalence class of metrically equivalent metric rays.
In locally finite graphs the notion of being an end and being a metric end coincide.
Thus for connected locally finite $k$-CS-transitive graphs ($k\ge3$) with more than one \emph{end} the automorphism group acts transitively on its metric ends.
In spite of the local ends this extends to graphs that are not locally finite.

\begin{Tm}
If $k\ge 3$, then the automorphism group of any $k$-CS-transitive graph with at least two ends acts transitively on the metric ends of the graph.\qed
\end{Tm}

\end{document}